\documentclass[11pt]{article}

\usepackage{amssymb} 
\usepackage{latexsym}
\usepackage{amsmath} 

\def\n{\nabla}

\def\F{{\mathcal F}}
\def\I{{\mathcal I}}
\def\R{{\mathcal R}}

\newenvironment{proof}{{\bf Proof.}}{\hfill$\rule{1ex}{1ex}$\par\medskip}
\newtheorem{theorem}{Theorem}[section]
\newtheorem{proposition}[theorem]{Proposition}

\newtheorem{remark}[theorem]{Remark}
\newtheorem{definition}[theorem]{Definition}
\newtheorem{lemma}[theorem]{Lemma}
\newtheorem{corollary}[theorem]{Corollary}
\newcommand{\bp}{\begin{proof}\;} \newcommand{\ep}{\end{proof}}

\newcounter{mycounter} 
\title{Tangent Lie algebras to the holonomy group of a Finsler manifold}

\author{Zolt\'an Muzsnay and P\'eter T. Nagy}

\date{Institute of Mathematics, University of Debrecen\\
  H-4010 Debrecen, Hungary, P.O.B. 12
  \\
  \bigskip {\it E-mail}: {\tt {}muzsnay@science.unideb.hu}, {\tt
    {}petert.nagy@science.unideb.hu}}

\begin{document}

\maketitle
\abstract Our goal in this paper is to make an attempt to find the largest 
Lie algebra of vector fields on the indicatrix such that all its elements 
are tangent to the holonomy group of a Finsler manifold. First, we introduce 
the notion of the \emph{curvature algebra}, generated by curvature vector 
fields, then we define the \emph{infinitesimal holonomy algebra} by the 
smallest Lie algebra of vector fields on an indicatrix, containing the 
curvature vector fields and their horizontal covariant derivatives with 
respect to the Berwald connection. At the end we introduce 
the notion of the \emph{holonomy algebra} of a Finsler manifold by all 
conjugates of infinitesimal holonomy algebras by parallel translations with 
respect to the Berwald connection. We prove that this holonomy algebra is 
tangent to the holonomy group. 

\section{Introduction}

The notion of the holonomy group of Riemannian manifolds can be
generalized very naturally for Finsler manifolds: it is the group
generated by canonical homogeneous (nonlinear) parallel translations
along closed loops. Until now the holonomy groups of non-Riemannian
Finsler manifolds were described only in special cases: the Berwald
manifolds have the same holonomy group as some Riemannian manifolds
(cf. Z.  I. Szab\'o, \cite{Sza}) and the holonomy groups of Landsberg
manifolds are compact Lie groups (cf. L. Kozma, \cite{Koz}). A thorough
study of the holonomy algebras of homogeneous (nonlinear) connections
was initiated by W. Barthel \cite{Bar}, he gave a successive extension
by Berwald's covariant derivation of the Lie algebras generated by the
curvature vector fields.  A general setting for the study of infinite
dimensional holonomy groups and holonomy algebras of nonlinear
connections was initiated by P. Michor in \cite{Mic1}, but the
tangential properties of the holonomy algebras to the holonomy group
were not clarified.
\\[1ex]
Recently, the authors introduced in \cite{Mu_Na} the notion of tangent
Lie algebras to the holonomy group and proved that the curvature algebra
(the Lie algebra generated by curvature vector fields) is a tangent
algebra to the holonomy group.  With this technique we have constructed
a Finsler manifold (with singular metric) with infinite dimensional
curvature algebra, which implies that the holonomy group can not be a
finite dimensional Lie group in this case.  We suspect that for most of
non-Riemannian Finsler manifolds, the holonomy group is not a
finite dimensional Lie group. \\[1ex]
In a recent paper \cite{Cr_Sa} M.~Crampin, D.J.~Saunders carried on a
deep analysis of the holonomy structures of bundles with fibre metrics,
and in particular the holonomy structures of Landsbergian type Finsler
manifolds. In these cases, the holonomy groups are finite dimensional
Lie groups.  They introduced the notion of holonomy algebra and proved a
version of Ambrose-Singer Theorem for such spaces.  Reflecting to our
results, they noticed that in the general Finslerian framework the
holonomy algebra should contain the parallel translated curvature
algebras.  They showed that in this case the topological closure of this
holonomy algebra contains the covariant derivatives of curvature vector
fields, but the tangent properties of the successive covariant
derivatives of curvature vector fields are not obvious from this approach
in the cases, when the holonomy group is not a finite dimensional Lie
group.  The difficulty comes from the fact, that a topologically
non-closed infinite dimensional Lie algebra of vector fields may expand, 
if we add the covariant derivatives of its elements.
\\[1ex]
Our goal in this paper is to make an attempt to find the right notion of
the holonomy algebra of Finsler spaces.  The holonomy algebra should be
the largest Lie algebra such that all its elements are tangent to the
holonomy group. In our attempt we are building successively Lie algebras
having the tangent properties. First, we introduce the notion of the
\emph{curvature algebra} (the Lie algebra generated by curvature vector
fields) which is a tangent Lie algebra to the holonomy group (c. f.
\cite{Mu_Na}).  Then we define the \emph{infinitesimal holonomy
  algebra} by the smallest Lie algebra of vector fields on an
indicatrix, containing the curvature vector fields and their horizontal
covariant derivatives with respect to the Berwald connection and prove
the tangential property of this Lie algebra to the holonomy group.  At
the end we introduce the notion of the \emph{holonomy algebra} of a
Finsler manifold by all conjugates of infinitesimal holonomy algebras by
parallel translations with respect to the Berwald connection. We prove
that this holonomy algebra is tangent to the holonomy group.  The
question of whether the holonomy algebra introduced in this way is the
largest Lie algebra, which is tangent to the holonomy group, is still
open.

\section{Preliminaries}

Let $M$ be an $n$-dimensional $C^\infty$ manifold and let ${\mathfrak X}^{\infty}(M)$ 
denote the vector space of smooth vector fields on $M$. For a local coordinate system 
$(x^1,\dots,x^n)$ on $M$ we denote by $(x^1,\dots,x^n;y^1,\dots,y^n)$ the induced local 
coordinate system on the tangent bundle $TM$. 

\subsubsection*{Finsler manifold, canonical connection, parallelism}

A \emph{Finsler manifold} is a pair $(M,\mathcal F)$, where the Finsler
function $\F\colon TM \to \mathbb{R}$ is continuous, smooth on $\hat T M
:= TM\setminus\! \{0\}$, its restriction ${\mathcal F}_x={\mathcal
  F}|_{_{T_xM}}$ is a positively homogeneous function of degree $1$ and
the symmetric bilinear form (the Finsler metric)
\begin{displaymath}
  g_{x,y} \colon (u,v)\ \mapsto \ g_{ij}(x, y)u^iv^j=\frac{1}{2}
  \frac{\partial^2 \mathcal F^2_x(y+su+tv)}{\partial s\,\partial
    t}\Big|_{t=s=0}
\end{displaymath}
is positive definite at every $y\in \hat T_xM$.
\\[1ex]
\emph{Geodesics} of Finsler manifolds are determined by a system of
second order ordinary differential equation $\ddot{x}^i + 2 G^i(x,\dot
x)=0$, $i = 1,...,n,$ where $G^i(x,\dot x)$ are locally given by
\begin{equation}
  \label{eq:G_i}  G^i(x,y):= \frac{1}{4}g^{il}(x,y)\Big(2\frac{\partial
    g_{jl}}{\partial x^k}(x,y) -\frac{\partial g_{jk}}{\partial
    x^l}(x,y) \Big) y^jy^k.
\end{equation}
The associated \emph{homogeneous (nonlinear) parallel translation}
$\tau_{c}:T_{c(0)}M\to T_{c(1)}M$ along a curve $c:[0,1]\to \R$ is
defined by vector fields $X(t)=X^i(t)\frac{\partial}{\partial x^i}$
along $c(t)$ which are solutions of the differential equation
\begin{equation} 
  \label{eq:D} 
  D_{\dot c} X (t):=\Big(\frac{d X^i(t)}{d
    t}+ G^i_j(c(t),X(t))\dot c^j(t) \Big)\frac{\partial}{\partial x^i}
  =0,\quad\text{where}\quad G^i_j=\frac{\partial G^i}{\partial y^j}.
\end{equation}

\subsubsection*{Horizontal distribution, Berwald connection, curvature}

Let $(TM,\pi ,M)$ and $(TTM,\rho,TM)$ denote the first and the second
tangent bundle of the manifold $M$, respectively. The horizontal
distribution ${\mathcal H}TM \!  \subset\! TTM$ associated with the
Finsler manifold $(M, \mathcal F)$ can be defined as the image of the
horizontal lift which is an isomorphism
\begin{math} 
  X \to X^h 
\end{math} 
between $T_xM$ and ${\mathcal H}_{y}TM$ at $y\in T_xM$ defined by
\begin{equation} 
  \label{eq:lift}
  \Big(X^i\frac{\partial}{\partial
    x^i}\Big)^{\! h}:= X^i\left(\frac{\partial}{\partial x^i}
    -G_i^k(x,y)\frac{\partial}{\partial y^k} \right).
\end{equation}
If ${\mathcal V}TM:= \mathrm{Ker} \, \pi_{*} \!  \subset\!  TTM$ denotes
the vertical distribution on $TM$, then for any $y\in TM$ we have $T_yTM
= {\mathcal H}_yTM \oplus {\mathcal V}_yTM$. The projectors
corresponding to this decomposition will be denoted by $h:TTM \to
{\mathcal H}TM$ and $v:TTM \to {\mathcal V}TM$.  We note that the
vertical distribution is integrable.
\\[1ex]
Let $(\hat{\mathcal V}TM,\rho,\hat T M)$ be the vertical bundle over
$\hat T M := TM\setminus\! \{0\}$. We denote by ${\mathfrak
  X}^{\infty}(M)$, respectively by $\hat{\mathfrak X}^{\infty}(TM)$ the
vector space of smooth vector fields on $M$ and of smooth sections of
the bundle $(\hat{\mathcal V}TM,\tau,\hat T M)$, respectively.  The
\emph{horizontal Berwald covariant derivative} of a section
$\xi\in\hat{\mathfrak X}^{\infty}(TM)$ by a vector field $X\in{\mathfrak
  X}^{\infty}(M)$ is $\nabla_X\xi := [X^h,\xi]$.
\\[1ex]
In an induced local coordinate system $(x^i,y^i)$ on $TM$ for vector
fields $\xi(x,y) = \xi^i(x,y)\frac {\partial}{\partial y^i}$ and $X(x) =
X^i(x)\frac {\partial}{\partial x^i}$ we have (\ref{eq:lift}) and hence
\begin{equation}
  \label{covder}
  \nabla_X\xi = \left(\frac{\partial\xi^i(x,y)}{\partial x^j}
    - G_j^k(x,y)\frac{\partial \xi^i(x,y)}{\partial y^k} + 
    \frac{\partial G_j^i(x,y)}{\partial y^k}(x,y)\xi^k(x,y)\right)
  X^j\frac {\partial}{\partial y^i}. 
\end{equation}
Let $(\pi^{*}TM,\bar{\pi},\hat T M)$ be the pull-back bundle of $(\hat T
M,\pi ,M)$ by the map $\pi:TM\to M$.  Clearly, the mapping
\begin{equation} \label{iden} (x,y,\xi^i\frac {\partial}{\partial
    y^i})\mapsto(x,y,\xi^i\frac {\partial}{\partial x^i})
  :\;\hat{\mathcal V}TM\rightarrow \pi^{*}TM
\end{equation} 
is a canonical bundle isomorphism.  In the following we will use the
isomorphism (\ref{iden}) for the identification
of these bundles. \\[1ex]
The \emph{Riemannian curvature tensor} field
\begin{math}
  R_{(x,y)}(X,Y):= v\big[X^h,Y^h\big],
\end{math}
$X,Y\in T_x M$, $(x,y)\in \hat T M$ characterizes the integrability of
the horizontal distribution. Namely, if the horizontal distribution
$\mathcal H \hat TM$ is integrable, then the Riemannian curvature is
identically zero.  The expression of the Riemannian curvature tensor
$R_{(x,y)} = R^i_{jk}(x,y)dx^j\otimes
dx^k\otimes\frac{\partial}{\partial x^i}$ on the pull-back bundle
$(\pi^{*}TM,\bar{\pi},\hat T M)$ is
\begin{displaymath}
  R^i_{jk}(x,y) =  \frac{\partial G^i_j(x,y)}{\partial x^k} 
  - \frac{\partial G^i_k(x,y)}{\partial x^j} + G_j^m(x,y)
  G^i_{k m}(x,y) - G_k^m(x,y) G^i_{j m}(x,y). 
\end{displaymath} 

\subsubsection*{Indicatrix bundle} 

The \emph{indicatrix} $\I_pM$ of an $n$-dimensional Finsler manifold
$(M,\mathcal F)$ at a point $p \in M$ is the compact hypersurface
$\I_pM:= \{y \in T_pM;\ \mathcal F(y)=1\}$ in $T_pM$, diffeomorphic to
the standard $(n-1)$-sphere. The \emph {indicatrix bundle} $(\I
M,\pi,M)$ of $(M,\mathcal F)$ is a smooth subbundle of the tangent
bundle $(TM,\pi ,M)$. The group ${\mathsf {Diff}}^\infty({\I}_pM)$ of
all smooth diffeomorphisms of an indicatrix ${\I}_pM$ is a regular
infinite dimensional Lie group modeled on the vector space ${\mathfrak
  X}^\infty({\I}_pM)$ of smooth vector fields on ${\I}_pM$.  The Lie
algebra of the infinite dimensional Lie group
${\mathsf{Diff}}^\infty(\I_pM)$ is the vector space ${\mathfrak
  X}^\infty({\I}_pM)$, equipped with the negative of the usual Lie
bracket, (c.f. A. Kriegl and P. W. Michor \cite{KrMi}, Section 43).
\\[1ex]
Let $c(t)$, $0\leq t\leq a$ be a smooth curve joining the points
$p\!=\!c(0)$ and $q\!=\!c(a)$ in the Finsler manifold $(M, \mathcal F)$.
Since the parallel translation $\tau_{c}:T_pM\to T_qM$ along the curve
$c:[0,a]\to M$ is a differentiable map between $\hat T_{p}M$ and $\hat
T_{q}M$ preserving the value of the Finsler function, it induces a
parallel translation $\tau_{c}\colon \I_pM \longrightarrow \I_qM $
in the indicatrix bundle. 

\subsubsection*{Holonomy group}

The notion of the holonomy group of Riemannian manifolds can be
generalized very naturally for Finsler manifolds:
\begin{definition} 
  \emph{The \emph {holonomy group} $\mathsf{Hol}(p)$ of a Finsler space
    $(M, \F)$ at $p\in M$ is the subgroup of the group of
    diffeomorphisms ${\mathsf{Diff}}^\infty({\I}_pM)$ of the indicatrix
    ${\I}_pM$ determined by parallel translation of ${\I}_pM$ along
    piece-wise differentiable closed curves initiated at the point $p\in
    M$.}
\end{definition}
Clearly, the holonomy groups at different points of $M$ are
isomorphic. We note that the holonomy group $\mathsf{Hol}(p)$ is a
topological subgroup of the regular infinite dimensional Lie group
${\mathsf {Diff}}^\infty(\I_pM)$, but its differentiable structure is
not known in general.

\section{Tangent Lie algebras to diffeomorphism groups}
\label{tangetial}

Here we discuss the tangential properties of Lie algebras of vector
fields to an abstract subgroup of the diffeomorphism group of a
manifold. The results of this section will be applied in the following
to the investigation of tangent Lie algebras of the holonomy subgroup of
the diffeomorphism group of an indicatrix ${\I}_xM$ and to the fibred
holonomy subgroup of the diffeomorphism group of the indicatrix bundle
${\I}(M)$.
\\[1ex]
Let $P$ be a $C^\infty$ manifold, let $H$ be a (not necessarily
differentiable) subgroup of the diffeomorphism group
$\mathsf{Diff}^{\infty}(P)$ and let ${\mathfrak X}^{\infty}(P)$ be the
Lie algebra of smooth vector fields on $P$. 

\begin{definition}
  \emph{A vector field $X\!\in\! {\mathfrak X}^{\infty}(P)$ is called
    \emph{tangent} to the subgroup $H$ of $\mathsf{Diff}^{\infty}(P)$,
    if there exists a ${\mathcal C}^1$-differentiable $1$-parameter
    family $\{\phi_t\in H\}_{t\in(-\varepsilon,\varepsilon)}$ of
    diffeomorphisms of $M$ such that
    \begin{math}
      \phi_{0}=\mathsf{Id}
    \end{math}
    and
    \begin{math}
      \frac{\partial\phi_t}{\partial t}\big|_{t=0}=X.
    \end{math}
    A Lie subalgebra $\mathfrak g$ of ${\mathfrak X}^{\infty}(P)$ is
    called \emph{tangent} to $H$, if all elements of $\mathfrak g$ are
    tangent vector fields to $H$.}
\end{definition} 
Unfortunately, it is not true, that tangent vector fields to the group
$H$ generate a tangent Lie algebra to $H$. This is why we have to
introduce a stronger tangency property in Definition
\ref{def:strongly_tangent}.
\begin{definition}
  \emph{A ${\mathcal C}^{\infty}$-differentiable $k$-parameter family
    \(\{\phi_{(t_1,\dots,t_k)}\in \mathsf{Diff}^{\infty}(P)\}_{t_i\in
      (-\varepsilon,\varepsilon)}\)
      of diffeomorphisms of $P$ is called a
    \emph{commutator-like family} if it satisfies the equations}
  \begin{displaymath}
    \phi_{(t_1,\dots,t_k)}=\mathsf{Id},\quad \text{\emph{whenever}}
    \quad t_j=0 \quad \text{\emph{for some}}\quad 1 \leq j\leq k.
  \end{displaymath}
\end{definition}
We remark, that the commutators of commutator-like families are
commutator-like, and the inverse of commutator-like families are
commutator-like.
\begin{definition}
  \label{def:strongly_tangent}
  \emph{A vector field $X\!\in\! {\mathfrak X}^{\infty}(P)$ is called
    \emph{strongly tangent} to the subgroup $H$ of
    $\mathsf{Diff}^{\infty}(P)$, if there exists a commutator-like
    family $\{\phi_{(t_1,\dots,t_k)}\in
    \mathsf{Diff}^{\infty}(P)\}_{t_i\in (-\varepsilon,\varepsilon)}$ of
    diffeomorphisms satisfying the conditions
    \begin{enumerate}
    \item[\em{(A)}] $\phi_{(t_1,\dots,t_k)}\in H\quad \text{for
        all}\quad t_i\in (-\varepsilon,\varepsilon),\quad 1\leq i\leq
      k,$
    \item[\em{(B)}] $\frac{\partial^k\phi_{(t_1,\dots,t_k)}}{\partial
        t_1\cdots\partial t_k}\big|_{(0,\dots,0)}=X.$
    \end{enumerate}}
\end{definition} 
It follows from the commutator-like property that
\begin{math}
  \frac{\partial^k\phi_{(t_1,\dots,t_k)}}{\partial t_1\cdots\partial
    t_k}\big|_{(0,\dots,0)}
\end{math}
is the first non-necessarily vanishing derivative of the diffeomorphism
family $\{\phi_{(t_1,...,t_k)}\}$ at any point $x\in P$, and therefore
it determines a vector field. On the other hand, by reparametrizing the
commutator like family of diffeomorphism, it can be shown that if a
vector field is strongly tangent to a group $H$, then it is also tangent
to $H$. Moreover, we have the following
\begin{theorem}
  \label{liealg} 
  Let $\mathcal V$ be a set of vector fields strongly tangent to the
  group $H\subset\mathsf{Diff}^{\infty}(P)$.  The Lie subalgebra
  $\mathfrak v$ of ${\mathfrak X}^{\infty}(P)$ generated by $\mathcal V$
  is tangent to $H$.
\end{theorem}
The proof of the theorem is based on two importent observation. The
first is a generalization of the well-known relation between the
commutator of vector fields and the commutator of their induced
flows. Namely, if $\{\phi_{(s_1,...,s_k)}\}$ and
$\{\psi_{(t_1,...,t_l)}\}$ are commutator-like $k$-parameter,
respectively \,$l$-parameter families of local diffeomorphisms, then the
family of (local) diffeomorphisms
$[\phi_{(s_1,...,s_k)},\psi_{(t_1,...,t_l)}]$ defined by the commutator
of the group $\mathsf{Diff}^{\infty}(U)$ is a commutator-like
$k+l$-parameter family and
\begin{displaymath}
  {\frac{\partial^{k+l}[\phi_{(s_1,...,s_k)},\psi_{(t_1,...,t_l)}]}
    {\partial s_1\;...\;\partial s_k\;\partial t_1\;...\;\partial t_l}
    \Big|_{(0,...,0;\,0,...,0)}}(x)=-\Bigg[{\frac{\partial^k\phi_{(s_1,...,s_k)}}
    {\partial s_1\;...\;\partial
      s_k}\Big|_{(0,...,0)}},{\frac{\partial^l\psi_{(t_1,...,t_l)}}
    {\partial t_1\;...\;\partial t_l}\Big|_{(0,...,0)}}\Bigg](x)
\end{displaymath}
at any point $x\in U$. The second important fact to prove the theorem is
that the linear combinations of vector fields tangent to $H$ are also
tangent to $H$.  The detailed computations can be found in \cite{Mu_Na}.

\section{The curvature algebra at a point}

Now, se summarize our results on the tangent Lie algebras of the holonomy group $\mathsf{Hol}(p)$ at a point
$p\in M$, their proofs can be found in \cite{Mu_Na}.
\begin{definition} 
  \label{atpoint}
  \emph{A vector field $\xi\in {\mathfrak X}({\I}_pM)$ on the indicatrix
    $\I_pM$ of the Finsler manifold $(M, \F)$ is called a
    \emph{curvature vector field at the point $p\in M$}, if it is in the
    image of the curvature tensor, i.e. if there exist $X,Y\in T_pM$
    such that $\xi = r_p(X,Y)$, where
    \begin{equation}
      \label{eq:r}
      r_p(X,Y)(y):=R_{(p,y)}(X^h,Y^h)
    \end{equation}
    The Lie subalgebra
    \begin{math} 
      {\mathfrak R}_p \! := \! \big\langle \, r_p(X,Y); \;X, Y \!\in\!
      T_pM \,\big\rangle
    \end{math}
    of ${\mathfrak X}({\I}_pM)$ generated by the curvature vector fields
    at the point $p\in M$ is called the \emph{curvature algebra at the
      point $p\in M$}.}
\end{definition} 
Since the Finsler function is preserved by parallel translations, its
derivatives with respect to horizontal vector fields are identically
zero. According to \cite{Shen1}, eq.\,(10.9), the derivative of the
Finsler metric with respect to $R_{(p,y)}(X^h ,Y^h)$ vanishes, i.e.
\begin{displaymath}
  g_{(p,y)}\big(y,R_{(p,y)}(X^h,Y^h)\big)=0, 
  \quad \text{for any} \quad y,X,Y\in T_xM.
\end{displaymath} 
This means that the curvature vector fields $\xi\!=\!r_p(X,Y)$ are
tangent to the indicatrix.  In the sequel we investigate the tangential
properties of the curvature algebra to the holonomy group of the
canonical connection $\nabla $ of a Finsler manifold.

\begin{proposition}
  \label{curvat}
  Any curvature vector field at a point $p\!\in\! M$ is strongly tangent
  to the holonomy group $\mathsf{Hol}(p)$.
\end{proposition}
\begin{proposition} 
  The curvature algebra ${\mathfrak R}_p$ at any point $p\in M$ of a
  Riemannian manifold $(M,g)$ is isomorphic to the linear Lie algebra on
  the tangent space $T_pM$ generated by the curvature operators of
  $(M,g)$ at $p\in M$.
\end{proposition}
\begin{remark} 
  The dimension of the curvature algebra at any point $p\in M$ of a
  Finsler surface is $\leq 1$.
\end{remark}

\section{Fibred holonomy group and fibred holonomy algebra}

Now, we introduce the notion of the fibred holonomy group of a Finsler
manifold $(M, \F)$ as a subgroup of the diffeomorphism group of the
total manifold $\I M$ of the bundle $(\I M,\pi,M)$ and apply our results
on tangent vector fields to an abstract subgroup of the diffeomorphism
group to the study of tangent Lie algebras to the fibred holonomy group.
\begin{definition}
  \emph{The} fibred holonomy group $\mathsf{Hol_f}(M)$ \emph{of $(M,
    \F)$ consists of fibre preserving diffeomorphisms
    $\Phi\in\mathsf{Diff}^{\infty}(\I M)$ of the indicatrix bundle $(\I
    M,\pi,M)$ such that for any $p\in M$ the restriction $\Phi_p
    =\Phi|_{\I_p M}\in\mathsf{Diff}^{\infty}(\I_p M)$ belongs to the
    holonomy group $\mathsf{Hol}(p)$.}
\end{definition}
We note that the holonomy group $\mathsf{Hol}(p)$ and the fibred
holonomy group $\mathsf{Hol_f}(M)$ are topological subgroups of the
infinite dimensional Lie groups ${\mathsf {Diff}}^{\infty}(\I_pM)$ and
${\mathsf {Diff}}^{\infty}(\I M)$ respectively.
\\[1ex]
The definition of strongly tangent vector fields yields
\begin{remark} 
  A vector field $\xi \in {\mathfrak X}^{\infty}(\I M)$ is strongly
  tangent to the fibred holonomy group $\mathsf{Hol_f}(M)$ if and only
  if there exists a family $\big\{\Phi_{(t_1,\dots,t_k)}
  \big|_{{\I}M}\big\}_{t_i\in (-\varepsilon,\varepsilon)}$ of fibre
  preserving diffeomorphisms of the bundle $(\I M,\pi,M)$ such that for
  any indicatrix ${\I}_p$ the induced family
  $\big\{\Phi_{(t_1,\dots,t_k)}
  \big|_{{\I}_pM}\big\}_{t_i\in(-\varepsilon,\varepsilon)}$ of
  diffeomorphisms is contained in the holonomy group $\mathsf{Hol}(p)$
  and $\xi\big|_{{\I}_pM}$ is strongly tangent to $\mathsf{Hol}(p)$.
\end{remark}
Since $\pi\big(\Phi_{(t_1,\dots,t_k)}(p)\big)\! \equiv \!p$ and
$\pi_*(\xi) \!=\! 0$ for every $p\in U$, we get the
\begin{corollary}\label{infholgen}
  Strongly tangent vector fields to the fibred holonomy group
  $\mathsf{Hol_f}(M)$ are vertical vector fields.  If $\xi \in
  {\mathfrak X}^{\infty}(\I M)$ is strongly tangent to
  $\mathsf{Hol_f}(M)$ then its restriction $\xi_p:= \xi\big|_{\I_p}$ to
  any indicatrix $\I_p$ is strongly tangent to the holonomy group
  $\mathsf{Hol}(p)$.
\end{corollary}  
The curvature vector fields and the curvature algebra at a point has
been defined on an indicatrix of the manifold $M$. Now we extend the
domain of their definition to the total manifold of the indicatrix
bundle.
\begin{definition}
  \emph{A vector field $\xi\in {\mathfrak X}^{\infty}({\I}M)$ on the
    indicatrix bundle $\I M$ is a \emph{curvature vector field} of the
    Finsler manifold $(M, \F)$, if there exist $X, Y\in {\mathfrak
      X}^{\infty}(M)$ such that $\xi = r(X,Y)$, where
    $r(X,Y)(x,y):=R_{(x,y)}(X_x,Y_x)$ for $x\in M$ and $y \in {\I}_xM$.
    \\
    The Lie algebra $\mathfrak{R}(M)$ generated by the curvature vector
    fields of $(M, \F)$ is called the \emph{curvature algebra} of the
    Finsler manifold $(M, \F)$.}
\end{definition}
\begin{proposition} \label{cvf} If the Finsler manifold $(M, \F)$ is 
diffeomorphic to $\mathbb{R}^n$ then any curvature vector field \;$\xi\in 
{\mathfrak X}^{\infty}({\I}M)$ of $(M, \F)$ is strongly tangent to the fibred 
holonomy group $\mathsf{Hol_f}(M)$.
\end{proposition}
\bp Since $M$ is diffeomorphic to $\mathbb{R}^n$ we can identify the
manifold $M$ with the vector space $\mathbb{R}^n$. Let $\xi= r(X,Y)\in
{\mathfrak X}^{\infty}({\I}\mathbb{R}^n)$ be a curvature vector field
with $X, Y\in {\mathfrak X}^{\infty}(\mathbb{R}^n)$. According to
Proposition \ref{curvat} its restriction $\xi\big|_{\I_p\mathbb{R}^n}$
to any indicatrix $\I_p\mathbb{R}^n$ is strongly tangent to the holonomy
groups $\mathsf{Hol}(p)$. We have to prove that there exists a family
$\big\{\Phi_{(t_1,\dots,t_k)} \big|_{{\I}\mathbb{R}^n}\big\}_{t_i\in
  (-\varepsilon,\varepsilon)}$ of fibre preserving diffeomorphisms of
the indicatrix bundle $(\I \mathbb{R}^n,\pi,\mathbb{R}^n)$ such that for
any $p\in \mathbb{R}^n$ the family of diffeomorphisms induced on the
indicatrix ${\I}_p$ is contained in $\mathsf{Hol}(p)$ and
$\xi\big|_{{\I}_p\mathbb{R}^n}$ is strongly tangent to $\mathsf{Hol}(p)$.\\
For any $p\in\mathbb{R}^n$ and $-1< s, t< 1$ let $\Pi(s X_p,t Y_p)$ be
the parallelogram in $\mathbb{R}^n$ determined by the vertexes $p, p + s
X_p, p + s X_p + t Y_p, p + t Y_p\in\mathbb{R}^n$ and let $\tau_{\,\Pi(s
  X_p,t Y_p)}:\I_p\to\I_p$ denote the (nonlinear) parallel translation
of the indicatrix $\I_p$ along the parallelogram $\Pi(s X_p,t Y_p)$ with
respect to the associated homogeneous (nonlinear) parallel translation
of the Finsler manifold $(\mathbb{R}^n, \F)$. Clearly we have
$\tau_{\,\Pi(s X_p,t Y_p)} = \mathsf{Id}_{\I \mathbb{R}^n}$, if $s=0$ or
$t=0$ and
\[ \frac{\partial^2\tau_{\,\Pi(s X_p,t Y_p)}}{\partial s\partial
  t}\Big|_{(s,t)=(0,0)} = \xi_p \quad\text{for every}\quad p
\in\mathbb{R}^n.\] Since $\Pi(s X_p,t Y_p)$ is a differentiable field of
parallelograms in $\mathbb{R}^n$, the maps $\tau_{\,\Pi(s X_p,t Y_p)}$,
$p\in\mathbb{R}^n$, $0< s, t< 1$, define a $2$-parameter family of fibre
preserving diffeomorphisms of the indicatrix bundle
$\I\mathbb{R}^n$. The diffeomorphisms induced by the family
$\big\{\tau_{\,\Pi(s X_p,t Y_p)}\big\}_{s,t\in (-1,1)}$ on any
indicatrix ${\I}_p$ are contained in $\mathsf{Hol}(p)$. Hence the vector
field $\xi\!\in\! {\mathfrak X}^{\infty}(\mathbb{R}^n)$ is strongly
tangent to the fibred holonomy group $\mathsf{Hol_f}(M)$, hence the
assertion is proved.  \ep
\begin{corollary} 
  If $M$ is diffeomorphic to $\mathbb{R}^n$ then the curvature algebra
  $\mathfrak{R}(M)$ of $(M, \F)$ is tangent to the fibred holonomy group
  $\mathsf{Hol_f}(M)$.
\end{corollary}
The following assertion shows that similarly to the Riemannian case, the
curvature algebra can be extended to a larger tangent Lie algebra
containing all horizontal covariant derivatives of the curvature algebra
vector fields.
\begin{proposition}
  \label{nabla_X} If $\xi \in {\mathfrak X}^{\infty}({\I}M)$ is strongly
  tangent to the fibred holonomy group $\mathsf{Hol_f}(M)$ of $(M, \F)$,
  then its horizontal covariant derivative $\nabla_{\!\! X}\xi$ along
  any vector field $X\in{\mathfrak X}^{\infty}(M)$ is also strongly
  tangent to $\mathsf{Hol_f}(M)$.
\end{proposition}
\bp Let $\tau$ be the (nonlinear) parallel translation along the flow
$\varphi$ of the vector field $X$, i.e. for every $p\in M$ and $t\in
(-\varepsilon_p,\varepsilon_p)$ the map $\tau_t(p)\colon\I_pM
\to\I_{\varphi_t(p)}M$ is the (nonlinear) parallel translation along the
integral curve of $X$. If $\{\Phi_{(t_1,\dots,t_k)}\}_{t_i\in
  (-\varepsilon,\varepsilon)}$ is a ${\mathcal
  C}^{\infty}$-differentiable $k$-parameter family
$\{\Phi_{(t_1,\dots,t_k)}\}_{t_i\in (-\varepsilon,\varepsilon)}$ of
fibre preserving diffeomorphisms of the indicatrix bundle $(\I M,\pi
|_M,M)$ satisfying the conditions of Definition \ref{infholgen} then the
commutator
\begin{displaymath} [\Phi_{(t_1, \ldots ,t_k)}, \tau_{t_{k+1}}]:=
  \Phi^{-1}_{(t_1,...,t_k)}\circ
  (\tau_{t_{k+1}})^{-1}\circ\Phi_{(t_1,...,t_k)}\circ \tau_{t_{k+1}}
\end{displaymath}
of the group $\mathsf{Diff}^{\infty}\big(\I M\big)$ fulfills $[\Phi_{(t_1,...,t_k)},\tau_{t_{k+1}}]=\mathsf{Id}$, 
if some of its variables equals $0$. Moreover
\begin{equation}\label{eq:phi_tau}
  \frac{\partial^{k+1}[\Phi_{(t_1...t_k)},\tau_{(t_{k+1})}]} {\partial t_1\;...\; \partial t_{k+1}} \Bigg|_{(0...0)} 
  = -\big[\xi, X^h\big]
\end{equation}
at any point of $M$, which shows that the vector field $\big[\xi,X^h\big]$ is strongly tangent to $\mathsf{Hol_f}(M)$.
Moreover, since the vector field $\xi$ is vertical, we have $h[X^h,\xi]=0$, and using $\nabla_X\xi := [X^h,\xi]$ we obtain
\begin{displaymath}
  -[\xi,X^h]=[X^h, \xi] =v[X^h, \xi] = \n_{X}\xi
\end{displaymath}
which yields the assertion.
\ep
\begin{definition}
  \emph{Let $\mathfrak{hol_f}(M)$ be the smallest Lie algebra of vector
    fields on the indicatrix bundle $\I M$ satisfying the properties
\begin{enumerate}
\item[(i)] any curvature vector field $\xi$ belongs to $\mathfrak{hol_f}(M)$,
\item[(ii)] if $\xi\in\mathfrak{hol_f}(M)$ and $X\in {\mathfrak
    X}^{\infty}(M)$, then the covariant derivative $\nabla_{\!\! X}\xi$
  also belongs to $\mathfrak{hol_f}(M)$.
\end{enumerate}
The Lie algebra $\mathfrak{hol_f}(M)\subset{\mathfrak
  X}^{\infty}({\I}M)$ is called the \emph{fibred holonomy algebra} of
the Finsler manifold $(M, \mathcal F)$.}
\end{definition}   
\begin{remark} 
  The fibred holonomy algebra $\mathfrak{hol_f}(M)$ is invariant with
  respect to the horizontal covariant derivation with respect to the
  Berwald connection, i.e.
  \begin{equation} 
    \label{eq:D_X_xi}
    \xi \in \mathfrak{hol_f}(M) \quad
    \text{and}\quad X\in {\mathfrak X}^{\infty}(M)\quad\Rightarrow \quad
    \nabla_{\!  X}\xi \in \mathfrak{hol_f}(M).
  \end{equation}
\end{remark}
The results of this sections yield the following 
\begin{theorem} If $M$ is diffeomorphic to $\mathbb{R}^n$ then the
  fibred holonomy algebra $\mathfrak{hol_f}(M)$ is tangent to the fibred
  holonomy group $\mathsf{Hol_f}(M)$.
\end{theorem}

\section{Infinitesimal holonomy algebra} 

Let $\mathfrak{hol_f}(M)\subset{\mathfrak X}^{\infty}({\I}M)$ be the
fibred holonomy algebra of the Finsler manifold $(M, \mathcal F)$ and
let $p$ be a a given point in $M$.
\begin{definition} 
  \emph{The Lie algebra $\mathfrak{hol}^{*}(p)\! := \!\big\{\, \xi_p \
    ;\ \xi\in\mathfrak{hol_f}(M)\,\big\} \subset{\mathfrak
      X}^{\infty}({\I}_pM)$ of vector fields on the indicatrix $\I_p M$
    is called the \emph{infinitesimal holonomy algebra at the point
      $p\in M$}.}
\end{definition} 
Clearly,  $\mathfrak{R}_p\subset\mathfrak{hol}^{*}(p)$ for any $p\in M$.\\
The following assertion is a direct consequence of the definition. It
shows that the infinitesimal holonomy algebra at a point $p$ of $(M,
\F)$ can be calculated in a neighbourhood of $p$.
\begin{remark}
  Let $(U, \mathcal F|_U)$ be an open submanifold of $(M, \F)$ such that
  $U\subset M$ is diffeomorphic to $\mathbb{R}^n$ and let $p\in U$. The
  infinitesimal holonomy algebras at $p$ of the Finsler manifolds $(M,
  \mathcal F)$ and $(U, \mathcal F|_U)$ coincide.
\end{remark}
Now, we can prove the following 
\begin{theorem} 
  The infinitesimal holonomy algebra $\mathfrak{hol}^{*}(p)$ at any
  point $p$ of the Finsler manifold $(M, \F)$ is tangent to the holonomy
  group $\mathsf{Hol}(p)$.
\end{theorem}
\begin{proof} 
  Let $U\subset M$ be an open submanifold of $M$, diffeomorphic to
  $\mathbb{R}^n$ and containing $p\in M$. According to the previous
  remark we have $\mathfrak{hol}^{*}(p)\! := \!\big\{\, \xi_p\ ; \ \xi
  \in \mathfrak{hol_f}(U) \,\big\}$.  Since the fibred holonomy algebra
  $\mathfrak{hol_f}(U)$ is tangent to the fibred holonomy group
  $\mathsf{Hol_f}(U)$ we obtain that $\mathfrak{hol}^{*}(p)$ is tangent
  to the holonomy group $\mathsf{Hol}(p)$.
\end{proof}

\section{Holonomy algebra} 

Let $x(t)$, $0\leq t\leq a$ be a smooth curve joining the points
$q\!=\!x(0)$ and $p\!=\!x(a)$ in the Finsler manifold $(M, \mathcal
F)$. If $y(t)=\tau_ty(0)\in {\I}_{x(t)}M$ is a parallel vector field
along $x(t)$, $0\!\leq\! t\!\leq\! a$, where $\tau_t:\I_qM\to
\I_{x(t)}M$ denotes the homogeneous (nonlinear) parallel translation,
then we have
\begin{math}
  D_{\dot x} y (t):=\Big(\frac{d y^i(t)}{d t}+ G^i_j(x(t),y(t))\dot
  x^j(t)\Big)\frac{\partial}{\partial x^i} =0.
\end{math}
Considering a vector field $\xi$ on the indicatrix ${\I}_qM$, the map
${\tau_a}_*\xi\circ \tau_a^{-1}: (p,y)\mapsto {\tau_a}_*\xi(y(a))$ gives
a vector field on the indicatrix $\I_pM$. Hence we can formulate
\begin{lemma}
  \label{conjalg} 
  For any vector field $\xi\in \mathfrak{hol}^{*}(q) \subset{\mathfrak
    X}^{\infty}({\I}_qM)$ in the infinitesimal holonomy algebra at $q$
  the vector field ${\tau_a}_*\xi\circ\tau_a^{-1} \!\in\! {\mathfrak
    X}^{\infty}({\I}_pM)$ is tangent to the holonomy group
  $\mathsf{Hol}(p)$.
\end{lemma}
\begin{proof} 
  Let $\{\phi_t\in \mathsf{Hol}(q)\}_{t\in(-\varepsilon,\varepsilon)}$
  be a ${\mathcal C}^1$-differentiable $1$-parameter family of
  diffeomorphisms of $\I_qM$ belonging to the holonomy group
  $\mathsf{Hol}(q)$ and satisfying the conditions $\phi_0 =\mathsf{Id}$,
  \; \( \frac{\partial\phi_t}{\partial t}\big|_{t=0}=\xi.\) Since the
  $1$-parameter family
  \begin{displaymath}
    {\tau_a}\circ\phi_t\circ \tau_a^{-1} 
    \in\mathsf{Diff}^{\infty}({\I}_pM)\}_{t\in(-\varepsilon,\varepsilon)}
  \end{displaymath}
  of diffeomorphisms consists of elements of the holonomy group
  $\mathsf{Hol}(p)$ and satisfies the conditions
  \begin{displaymath}
    {\tau_a}\circ\phi_0\circ \tau_a^{-1}
    =\mathsf{Id}, \quad \quad  \frac{\partial \big(\tau_a\circ\phi_t
      \circ \tau_a^{-1}\big)}{\partial t}\Big|_{t=0}
    ={\tau_a}_*\xi\circ \tau_a^{-1},
  \end{displaymath}
  the assertion follows.
\end{proof}
\begin{definition}
  \emph{A vector field $\mathbf{B}_\gamma\xi \in {\mathfrak
      X}^{\infty}({\I}_pM)$ on the indicatrix ${\I}_pM$ will be called}
  the Berwald translate \emph{of the vector field $\xi\in {\mathfrak
      X}^{\infty}({\I}_qM)$ along the curve $\gamma=x(t)$ if}
  \begin{displaymath}
    \mathbf{B}_\gamma\xi = {\tau_a}_*\xi\circ({\tau_a})^{-1}.
  \end{displaymath}
\end{definition}   
\begin{remark}
  \emph{Let $y(t)=\tau_ty(0)\in {\I}_{x(t)}M$ be a parallel vector field
    along $\gamma\!=\!x(t)$, $0\leq t\leq a$, started at $y(0)\in
    \I_{x(0)}M$. Then, the vertical vector field $\xi_t=\xi(x(t),y(t))$
    along $(x(t),y(t))$ is the Berwald translate \(\xi_t =
    \tau_{t*}\xi_0\circ{\tau_t}^{-1}\) if and only if}
  \begin{displaymath}
    \nabla_{\dot x}\xi = \left(\frac {\partial\xi^i(x,y)}{\partial x^j} 
      - G_j^k(x,y)\frac{\partial \xi^i(x,y)}{\partial y^k} + 
      G^i_{j k}(x,y)\xi^k(x,y)\right){\dot x}^j\frac {\partial}{\partial
      y^i}=0.
  \end{displaymath}
\end{remark}
Now, lemma \ref{conjalg} yields the following
\begin{corollary}
  If $\xi\!\in\! \mathfrak{hol}^{*}(q)$ then its Berwald translate
  $\mathbf{B}_\gamma\xi\!\in\! {\mathfrak X}^{\infty}({\I}_pM)$ along
  any curve $\gamma\!=\!x(t)$, $0\leq t\leq a$, joining $q\!=\!x(0)$ with
  $p\!=\!x(a)$ is tangent to the holonomy group $\mathsf{Hol}(p)$.
\end{corollary} 
This last statement motivates the following 
\begin{definition}
  \emph{The} holonomy algebra $\mathfrak{hol}_p(M)$ \emph{of the Finsler
    manifold $(M, \mathcal F)$ at the point $p\in M$ is defined by the
    smallest Lie algebra of vector fields on the indicatrix $\I_p M$,
    containing the Berwald translates of all infinitesimal holonomy
    algebras along arbitrary curves $x(t)$, $0\leq t\leq a$ joining any
    points $q=x(0)$ with the point $p=x(a)$.}
\end{definition} 
Clearly, the holonomy algebras at different points of the Finsler
manifold $(M, \mathcal F)$ are isomorphic. The previous lemma and
corollary yield the following
\begin{theorem}  
  The holonomy algebra $\mathfrak{hol}_p(M)$ at $p\in M$ of a Finsler
  manifold $(M, \F)$ is tangent to the holonomy group $\mathsf{Hol}(p)$.
\end{theorem}

\end{document}